\documentclass{article}
\usepackage{mynotes, mathtools,hyperref, indentfirst}
\usepackage[
    backend=biber, 
    style=alphabetic,
    url=false,    
    doi=true,     
    eprint=true   
]{biblatex}
\title{Whittaker Functionals for Exceptional Representations}
\author{Sihan Li}
\date{}
\begin{document}
\maketitle
\begin{abstract}
In the study of automorphic forms on covering groups, Kazhdan-Patterson gives a formula that 
determines when the Whittaker functionals of induced representations can be descended to exceptional
representations.  
In this paper, first, following the original approach given by Kazhdan-Patterson and applying the block-compatible 
2-cocycle of Banks-Levy-Sepanski, we fill in several details omitted in the original proof. 
Secondly, we compare the formula of Kazhdan-Patterson with that of McNamara and Gao, which is formulated in a more general setting.
\end{abstract}

\tableofcontents
\newpage 
\section{Introduction}
\subsection{Background} 

In the representation theory of reductive $p$-adic groups, Whittaker models serve as the foundation for the construction of local $L$-functions.
Let $\mathbb{G}$ be a quasi-split reductive group over a non-Archimedean local field $F$, and $G$ be its $F$-points $\mathbb{G}(F)$,
it is well-known that for an irreducible admissible $G$-representation $\pi$ and a generic additive character $\psi$ on $F$,
the space of generic $\psi$-Whittaker functionals on $\pi$ is at most one-dimensional.
This crucial fact is usually referred to as ``(local) multiplicity one theorem'',
which plays an important role in the study of automorphic $L$-functions. 

Things become more complicated for covering groups of $G$. For example, let the group of $n$-th roots of unity $\mu_n \subseteq F$, when $\mathbb{G}=\mathrm{GL}_r$, 
the uniqueness of Whittaker models usually
fails for genuine representations of its (twisted) Kazhdan-Patterson covering group $\widetilde{G}^{(c)}$ over $F$, with $c \in \mathbb{Z} / n \mathbb{Z}$. 
It is a basic problem to study 
the dimension of the space of Whittaker functionals $\mathrm{Wh}(V)$ associated with an admissible $\widetilde{G}^{(c)}$-representation $(\pi,V)$. 

In the study of representation theory of $\widetilde{G}^{(c)}$, a class of representations called exceptional representations are of great importance, 
denoted by $V_0(\omega^\prime)$, where $\omega^\prime$ is a character on a subgroup of $\widetilde{G}^{(c)}$.
They occur as local components in the restricted tensor product decomposition of the automorphic representation $(\pi_\theta, V_\theta)$ of global Kazhdan-Patterson 
covering group $\widetilde{G(\mathbb{A}_k)}$ in \cite[Theorem~II.2.1]{KP84}, motivated by seeking analogues of classical theta functions. 

Since $V_0(\omega^\prime)$ occur as quotients of principal series representations $V(\omega^\prime)$, any $\lambda \in \mathrm{Wh}(V_0(\omega^\prime))$
may be viewed as an element in $\mathrm{Wh}(V(\omega^\prime))$ satisfying a formula, see \cite[Corollary~I.3.4]{KP84} or Theorem~\ref{formula}. 
Moreover, the formula is used in \cite[Theorem~I.3.5]{KP84} to connect the number of free orbits of the Weyl group acting on a certain 
space with $\dim(\mathrm{Wh}(V_0(\omega^\prime)))$. As a corollary,
they obtain the equivalent conditions of $\dim(\mathrm{Wh}(V_0(\omega^\prime)))=1$.

Later, in \cite{GaoFan2017Dtrf}, Gao gives effective lower and upper bounds for $\dim(\mathrm{Wh}(V_0(\omega^\prime)))$ for certain 
Brylinski-Deligne covering group of a split linear algebraic group $\mathbb{G}$ over $F$ using an analogue of \cite[Corollary~I.3.4]{KP84} in the Brylinski-Deligne setting. 
Furthermore, for some covering groups of
certain semisimple simply connected $\mathbb{G}$, Gao gives necessary and sufficient conditions for the exceptional representations
to have a unique Whittaker functional.

\subsection{Main Results}
In this note, we focus on the formula of Whittaker functionals associated with exceptional representations 
for $\widetilde{G}^{(c)}= \widetilde{\mathrm{GL}_r(F)^{(c)}}$.
We provide two derivations of the formula. The first (Section \ref{calculation}) follows 
the strategy of \cite[\S~I.3]{KP84}. We fill in some details and compute the group multiplication
using the block-compatible cocycle of \cite{Banks1999BlockcompatibleMC}. The second (Section \ref{compare}) is obtained by 
specializing the formula given in \cite[Corollary~3.7]{GaoFan2017Dtrf} for Brylinski-Deligne covering groups to the
$\mathrm{GL}_r$ case. We briefly explain the difference between the two methods next.

The notations introduced next will be valid throughout the note. Let $F$ be a non-Archimedean local field containing all $n$-th roots of unity $\mu_n$, and let $G = \mathrm{GL}_r(F)$. 
We focus on the $n$-fold (twisted) Kazhdan-Patterson covering group $\widetilde{G}:=\widetilde{G}^{(c)}$, with the parameter $c\in \mathbb{Z}/ n\mathbb{Z}$ ignored if there is no confusion. 
In the foundational work of Kazhdan and Patterson \cite{KP84}, the group $\widetilde{G}$ was constructed using Matsumoto's method, 
which is rather abstract at least for the author.
In this note, we take a different approach to define the Kazhdan-Patterson covering group.

Because $\widetilde{G}$ is a central extension satisfying the exact sequence
\[
    0 \longrightarrow \mu_n \xrightarrow{i} \widetilde{G} \xrightarrow{p} G \longrightarrow 0
\]
we may alternatively construct it as a group extension via an explicit 2-cocycle $\sigma:=\sigma^{(c)}$. 
For $r=2$, Kubota constructed a concrete 2-cocycle for $\mathrm{SL}_2$ \cite{Kub67} and then for $\mathrm{GL}_2$ \cite{Kub69}. 
Later, Banks, Levy, and Sepanski \cite{Banks1999BlockcompatibleMC} obtained a 2-cocycle for $\widetilde{G}$ associated to general $r$.
Crucially, they proved that this 2-cocycle satisfies a block-compatible formula when restricted to any Levi subgroup. 
We will use this formula to calculate group multiplications. This approach was already used in \cite{BBL03}, but details were omitted.

From a more modern perspective, the Kazhdan-Patterson covering group, along with many other interesting topological central extensions, 
can be studied under the uniform framework of Brylinski-Deligne covering groups \cite{BD01}. Much of the work in \cite{KP84} relevant to 
Whittaker models has been generalized to Brylinski-Deligne covering groups for a connected reductive group $\mathbb{G}$ by McNamara \cite{Mc11}, 
and further for split $\mathbb{G}$ by Gao \cite{GaoFan2017Dtrf}. 
In particular, McNamara \cite{Mc11} proved an analogue of \cite[Lemma~I.3.3]{KP84} for Brylinski-Deligne covering groups under certain conditions.
And later, his proof was refined by Gao \cite{GaoFan2017Dtrf}, showing that those conditions may be dropped in Gao's setup. 
We will introduce the formula in \cite[Corollary~3.7]{GaoFan2017Dtrf}. By specifying some data associated to Brylinski-Deligne covering groups,
we may recover Kazhdan-Patterson covering groups and compare Theorem~\ref{formula} and \cite[Corollary~3.7]{GaoFan2017Dtrf}.
In fact, the two results coincide when the character $\omega^\prime$ is unramified and $\lvert n\rvert=1$.

\subsection{Applications to Co-period Integrals}
The period integral is a basic object in the relative Langlands program \cite{SY}. The Whittaker functional formula has direct applications to the study of period integrals on metaplectic covers. 
Following the work of Ginzburg-Jiang-Rallis \cite{Ginzburg2001NonvanishingOT}, recently, Cai, Fan, and She \cite{cai2025coperiodscentralsymmetriccube} investigated the global co-period integral attached to an automorphic form on $\mathrm{GL}_2$ 
and two exceptional theta series on the cubic Kazhdan-Patterson cover of $\mathrm{GL}_2$. 
In their work, they propose an Ichino-Ikeda type conjecture relating the global co-period to the central critical value of the symmetric cube $L$-function, $L(1/2, \sigma, \mathrm{Sym}^3)$,
and prove it in the case of Eisenstein series.

A critical step in the proof is the evaluation of the local unramified co-period integral. 
By utilizing the Whittaker functional formula established in \cite[Corollary~I.3.4]{KP84}, for each non-Archimedean place $v$,
one can relate the local integral of matrix coefficients in the unramified case with a certain $L$-value ratio \cite[Theorem~2.24]{cai2025coperiodscentralsymmetriccube}, 
\[
    I(f_v \otimes f_{+,v} \otimes f_{-,v}, f_v^\vee \otimes f_{+,v}^\vee \otimes f_{-,v}^\vee) = 
        \frac{L(1/2,\sigma_v,\mathrm{Sym}^3)}{L(1,\sigma_v,\mathrm{ad})}
\]
here $I$ can be decomposed into the product of two local co-period integrals.
This shows that local co-periods are compatible with the Ichino-Ikeda type conjecture.

\subsection{Future Directions}

The recent work \cite{GaoFan2017Dtrf} assumes the inducing character is unramified. It is an interesting problem to 
study the ramified case and see if similar conclusions hold for a connected split reductive group $\mathbb{G}$.

Another direction is to explore the case $\lvert n \rvert \neq 1$. If $\lvert n \rvert = 1$, for certain covers such as saturated covers, 
we expect that $\dim(\mathrm{Wh}(V_0(\omega^\prime)))$ is equal to the number of free orbits of a certain Weyl action. 
Studying this connection in the $\lvert n \rvert \neq 1$ case would be a helpful addition to the current theory.

Finally, for certain covers discussed above, the dimension of $\mathrm{Wh}(V_0(\omega^\prime))$ 
does not depend on the additive character $\psi$. 
However, for general covers, the dimension indeed depends on $\psi$, and this needs a better explanation. 
Understanding how $\psi$ changes this dimension is also a natural problem.
\\

\textbf{Acknowledgements.}
The author would like to express his sincere gratitude to Prof.\ Li Cai and Prof.\ Yangyu Fan for their invaluable guidance, patience, and support throughout the course of this project.
Also, the author is grateful to Prof.\ Fan Gao and Prof.\ Wen-Wei Li for their illuminating comments and advice.

\section{Main Results} \label{setup}
\subsection{Kazhdan-Patterson Covering Group}
Let 
\begin{itemize}
    \item $F$ be a local field containing all $n$-th roots of unity $\mu_n$, $n \geq 1$, embedded in $\mathbb{C}$ via $\varepsilon:\mu_n \rightarrow \mathbb{C}^\times$
    \item $O, \varpi, q$ be the valuation ring, uniformizer, cardinality of residue field of $F$, with absolute value $\lvert \cdot \rvert$ normalized as $\lvert \varpi \rvert = q^{-1}$
    \item $e_0$ be an additive character on $F$ with conductor $O$ and self-dual Haar measure $dx$ on $F$ with respect to $e_0$, $d^\times x = dx / \lvert x \rvert$
    \item $(-,-):F^\times / F^{\times n}  \times F^\times / F^{\times n} \rightarrow \mu_n$ be the $n$-th Hilbert symbol
    \item $c \in \Z / n \Z$
    \item \(d=\frac{n}{\gcd(n,r-1+2rc)}\)
    \item $G = \textnormal{GL}_r(F)$, $Z$ be the center, $N=N_+$ be the upper unipotent subgroup, $T$ be the diagonal torus, $T_n=\mathrm{diag}\{F^{\times n},\cdots,F^{\times n}\}$, $B$ be the upper triangular Borel subgroup
    \item $\alpha$ be a simple root $(i, i+1)$ of $G$, $s$ be the simple reflection associated to $\alpha$
    \item $(a, b)$ be a root of $G$, for any $h=\mathrm{diag}\{h_1, \cdots, h_r\} \in T$, define $h^{(a, b)}= h_a / h_b$
    \item $\mu:T \rightarrow \mathbb{R}^{>0}$ be the character defined as \[\mu(h)=\prod_{\beta \text{ is a positive root}} \lvert h^{\beta} \rvert^{1/2}\]
    \item $W$ be the Weyl group of $G$, consisting of all permutation matrices corresponding to the permutation group of $r$ letters $S_r$
    \item $w_0 \in W$ be the longest element, that is $w_0 = (\delta_{i+j,r+1})_{ij}$, where $\delta$ is the Kronecker symbol
\end{itemize}
Note that when $r=2$, $\mu$ is nothing but the square root of the modulus character of $B$. \\
Let \(\widetilde{G} := \widetilde{G}^{(c)}\) be the central extension
\[
0 \longrightarrow \mu_n \xrightarrow{i} \widetilde{G}^{(c)} \xrightarrow{p} G \longrightarrow 0
\]
equipped with group structure
\[
(g_1, z_1) \cdot (g_2, z_2) = \bigl(g_1 g_2, z_1 z_2 \sigma(g_1, g_2)\bigr), \quad g_1, g_2 \in G,\ z_1, z_2 \in \mu_n;
\]
with \(i(z) := (1, z)\), \(p(g, z) := g\) and \(\mathbf{s}(g) := (g, 1)\). 
Using the terminology of group cohomology, we fix a special 2-cocycle $\sigma:G\times G \rightarrow \mu_n$ lying in $Z^2(G,\mu_n)$, with $\mu_n$ viewed as a trivial $G$-module.
The explicit construction of the cocycle $\sigma$ has been done in \cite{Banks1999BlockcompatibleMC}, we cite a useful formula here.
\begin{theorem} \label{cocycle}
    Let $\beta$ be a composition of positive integers of $r$,
    which means that $\beta = (r_1, \dots, r_k)$ for certain $k$ and positive integers $r_i$ such that 
    $r = r_1 + \dots + r_k$. Let $G_\beta$ be the standard Levi subgroup of $G$ with respect to $\beta$. 
    And thus $G_\beta \simeq \mathrm{GL}_{r_1}(F) \times \dots \times \mathrm{GL}_{r_k}(F)$ with each $\mathrm{GL}_{r_i}(F)$ being regarded as a subgroup of $G$ via the 
    embedding $g_i \mapsto \operatorname{diag}\{I_{r_1}, \dots, I_{r_{i-1}}, g_i, I_{r_{i+1}}, \dots, I_{r_k}\}$.
    Then the 2-cocycle $\sigma$ satisfies the following description.
    \begin{itemize}
        \item \cite[Corollary~8]{Banks1999BlockcompatibleMC} When $r = 1$, we have $\sigma(x, y) = (x, y)^c$ for $x, y \in F^\times$.
    
        \item \cite[Theorem~11]{Banks1999BlockcompatibleMC} The restriction of $\sigma$ to $G_\beta \times G_\beta$ is given by
            \begin{align*}
                &\sigma\bigl(\operatorname{diag}\{x_1, \dots, x_k\}, \operatorname{diag}\{y_1, \dots, y_k\}\bigr)  \\
                =&\prod_{i=1}^k \sigma(x_i, y_i) \prod_{1 \leq i < j \leq k} (\det(x_i), \det(y_j))^{c+1} (\det(x_j), \det(y_i))^c,
            \end{align*}
        where $x_i, y_i \in \mathrm{GL}_{r_i}(F)$
    \end{itemize}
\end{theorem}
We shall adopt the convention that $c$ is taken to be even when $n \equiv 2 \pmod{4}$ without loss of generality 
by \cite[p.~41, Remarks]{KP84}. (If $c$ is odd, take $c^\prime = c+ n/2$.) Under this convention $\mathbf{s}$ is a homomorphism on $W$, and thus,
via \(\mathbf{s}\), we may view \(W\) and \(N\) as subgroups of \(\widetilde{G}\). Also, when $\lvert n\rvert=1$, by \cite[p.~43]{KP84}, there exists
a canonical lifting $\kappa^\ast:\mathrm{GL}_r(O) \rightarrow \widetilde{G}^{(c)}$.
\subsection{Representations and Whittaker Models}
We review the basics of exceptional representations and Whittaker models in this subsection. Our main reference is \cite{KP84}.
For any subgroup \(H \subset G\), let \(\widetilde{H} := p^{-1}(H)\). In the following, we assume $|n|=1$. \\
Let 
\begin{itemize}
    \item \(Z(\widetilde{G}):=\widetilde{Z}^{(c)}\) be the center of \(\widetilde{G}\)
    \item $\Gamma_1=O^\times F^{\times n}, \Gamma_2=O^\times F^{\times d}$
    \item $T_*=\left\{
\operatorname{diag}(\delta t_1,\ldots,\delta t_r):
t_i\in\Gamma_1,\ \delta\in\Gamma_2
\right\},$
    \item $N_+^\ast = \mathbf{s}(N_+)$
    \item \(\widetilde{B}_* = \widetilde{T}_* N_+^\ast\)
\end{itemize}
Note that \(\widetilde{T}_* \subset \widetilde{T}\) be a maximal abelian subgroup containing $\widetilde{T}_n \widetilde{Z}^{(c)}$.
For any \textbf{genuine} character \( \omega \) of \( \widetilde{Z}^{(c)} \widetilde{T}_n \), that is, $\omega \circ i = \varepsilon$, take an arbitrary extension  \( \omega^\prime \) to \( \widetilde{T}_* \). 
Also we may extend it to $\widetilde{B}_\ast$ by \[\omega^\prime(hn)=\omega^\prime(h), \quad h \in \widetilde{T}_\ast, n \in N_+^\ast\]
The following definition concerns the convergence of certain integrals.
\begin{definition}
    Let $\omega$ be as above. Write 
    \[
        \lvert \omega(h) \rvert = \prod_{i} \lvert h_i \rvert^{t_i}, \quad h \in \widetilde{T}_n \widetilde{Z}^{(c)}
    \]
    where $p(h)= \mathrm{diag}\{h_1,\cdots,h_r\}$ and $t_1, t_2, \cdots, t_r \in \mathbb{R}$ are uniquely determined by $\omega$. \\
    We say that $\omega$ is \textbf{dominant} if $t_1 > t_2 > \cdots > t_r$.
\end{definition}
Let \( V(\omega^\prime) \) be the usual smooth induction
\[
\left\{ f : \widetilde{G} \to \mathbb{C} \ \text{smooth} \mid f(bg) = \omega^\prime(b)\mu(b)f(g),\ b \in \widetilde{B}_*,\ g \in \widetilde{G} \right\}.
\]
Here \( V(\omega^\prime) \) depends only on \( \omega \) (independent of the extension).
For convenience, set 
\begin{itemize}
    \item $\omega^n(h)=\omega(\mathbf{s}(h^n))$, $h \in T$
    \item $\omega^n_\alpha(x) = \omega^n(\mathrm{diag}\{1,\cdots,1,x,x^{-1},1,\cdots,1\}), x \in F^\times$, where $x,x^{-1}$ lie at the $i$-th, $(i+1)$-th position respectively. This is a character on $F^\times$.
\end{itemize}
\begin{definition}
    We say that $\omega$ is \textbf{exceptional} if $\omega_\alpha^n(x) = \lvert x \rvert$, for any positive simple root $\alpha$ of $G$ and $x \in F^\times$.
\end{definition}

Fix a character \(\omega_0\) of \(\widetilde{Z}^{(c)}\), and form the character \(\omega_0 \times e\) of \(\widetilde{Z}^{(c)}N_+^\ast\) by
\[
(\omega_0 \times e)(zn) = \omega_0(z) e(n) \quad (z \in \widetilde{Z}^{(c)},\ n \in N_+^\ast).
\]
where $e(n)= e_0 (\sum_{1 \leq i \leq r-1} n_{i,i+1}), \forall n \in N_+$.
We shall suppose that \(\omega_0 \circ i = \varepsilon\) on \(\mu_n(F)\).
Then we form the representation \((\rho_e, L_e)\) of \(\widetilde{G}\) by smooth induction from \(\omega_0 \times e\).

\begin{definition}
    We say that a Whittaker model of an admissible $\widetilde{G}$-representation \((\pi, V)\) is an injective $\widetilde{G}$-morphism \((\pi, V) \to (\rho_e, L_e)\).
\end{definition} 
Let us recall how these are in general constructed. 
Let \((\pi', V')\) be the dual representation to \((\pi, V)\), i.e. \(V'\) is the dual space of \(V\), the action of \(v' \in V'\) on \(v \in V\) being denoted by \(\langle v', v \rangle\). Then \(\pi'\) is defined by
\[
\langle \pi'(g) v', v \rangle = \langle v', \pi(g)^{-1} v \rangle.
\]

Now suppose that \(\lambda \in V'\) and satisfies
\[
\pi'(n) \lambda = \overline{e(n)} \lambda, \quad n \in N_+^\ast
\]
and that \(\pi|_{\widetilde{Z}^{(c)}} = \omega_0\). Then if \(v \in V\), \((g \mapsto \langle \lambda, \pi(g) v \rangle) \in L_e\). If \((\pi, V)\) is irreducible and \(\lambda \neq 0\) this is clearly injective. 
Conversely, if \(t: (\pi, V) \to (\rho_e, L_e)\) then \(\lambda(v) = t(v)(I)\) is an element of \(V'\) as above. We denote the vector space of such \(\lambda\) by \(\mathrm{Wh}(V)\).\\
Now we introduce the \textbf{exceptional representations}. \\
Consider the character $\omega^\prime:\widetilde{B}_\ast \rightarrow \mathbb{C}^\times$ with previous assumptions, assuming it is dominant
we form the smooth induction $V(\omega^\prime) = I_{\widetilde{B}_\ast}^{\widetilde{G}}\omega^\prime$.
Let $N_-$ be the lower triangular unipotent subgroup, and put
\[
N_+(w)=N_+\cap wN_-w^{-1}, \qquad N_+^\ast(w)=\mathbf{s}(N_+(w)).
\]
For any $w \in W$ of $G$, it acts on $\omega^\prime(x)$ by $\prescript{w}{}{\omega^\prime}(x) = \omega^\prime(w^{-1}xw)$. We may construct an intertwining operator
\[
    \mathrm{I}_w : V(\omega^\prime) \rightarrow V(\prescript{w}{}{\omega^\prime})
\]
defined as
\[
    (\mathrm{I}_w f)(g) = \int_{N_+^\ast(w)} f(w^{-1} ng)dn, \quad \forall f \in V(\omega^\prime)
\]
and it converges when $\omega^\prime$ is dominant. 
By the regularization process (see \cite[p.~66]{KP84}), we may drop the dominant assumption.
In particular, when we take $w = w_0$ the longest element in $W$ and $\omega^\prime$ is \textbf{exceptional}, there is an exact sequence
    \begin{equation*} \label{exact_sequence}
        0 \longrightarrow \mathrm{Ker}(\mathrm{I}_{w_0}:V(\omega^\prime) \rightarrow V(\prescript{w_0}{}{\omega^\prime})) \longrightarrow V(\omega^\prime) \xrightarrow{\mathrm{I}_{w_0}} V_0(\omega^\prime) \longrightarrow 0
    \end{equation*}
where
\begin{itemize}
    \item $V_0(\omega^\prime)=\mathrm{Im}(\mathrm{I}_{w_0}:V(\omega^\prime) \rightarrow V(\prescript{w_0}{}{\omega^\prime}))$
    \item $\mathrm{Ker}(\mathrm{I}_{w_0}:V(\omega^\prime) \rightarrow V(\prescript{w_0}{}{\omega^\prime}))$ is generated by $\{\mathrm{Im}(\mathrm{I}_{s}:V(\prescript{s}{}{\omega^\prime}) \rightarrow V(\omega^\prime))\}_s$ 
    where $s$ runs through all the simple reflections of $W$
\end{itemize}
We call the representation $V_0(\omega^\prime)$ the \textbf{exceptional representation}.
\begin{definition}
    Suppose the character $\omega^\prime$ defined earlier is \textbf{dominant}, define a linear functional of $V(\omega^\prime)$, say $\lambda_\eta$, as
    \[
        \langle \lambda_\eta , f \rangle = \int_{N_+^\ast} \overline{e(n)} f(\eta w_0^{-1}n) dn
    \]
    where $f \in V(\omega^\prime), \eta \in \widetilde{T}$ and $w_0$ is the longest element in Weyl group $W$. \\
    The condition $\omega^\prime$ is dominant implies that the integral above is absolutely convergent, and easy to verify that
    \[\lambda_{h\eta}= (\omega^\prime \mu)(h)\lambda_\eta, \forall h \in \widetilde{T}_\ast\]
    so $\lambda_\eta \in \mathrm{Wh}(V(\omega^\prime))$.
\end{definition}
This suggests
\begin{proposition} \cite[Lemma~I.3.2]{KP84}
    One has
    \[
        \dim(\mathrm{Wh}(V(\omega^\prime))) = [\widetilde{T} : \widetilde{T}_\ast]
    \]
    and we have an isomorphism
        \begin{align*}
            \mathbb{Z}^r / \Lambda &\simeq \widetilde{T} / \widetilde{T}_\ast \\
            \mathbf{f} = (f_1,f_2,\cdots,f_r) &\mapsto \varpi^{\mathbf{f}} = \mathbf{s}\bigl(\mathrm{diag}\{\varpi^{f_1},\cdots,\varpi^{f_r}\}\bigr)
        \end{align*}
        here \[\Lambda = \left\{ \mathbf{f} \in \mathbb{Z}^r : f_j \equiv (2c + 1) \sum_{k} f_k \pmod{n} \right\}\]
    Moreover, if $\omega$ is dominant, then $\{\lambda_\eta\}_{\eta \in \widetilde{T}_\ast \backslash \widetilde{T}}$ is a basis of $\mathrm{Wh}(V(\omega^\prime))$.
\end{proposition}
So if $\omega$ is dominant, for any $w \in W$, consider a linear functional on $V(\omega^\prime)$
\[
    f \mapsto \langle \lambda_\eta, \mathrm{I}_w f\rangle
\]
where $\lambda_\eta \in \mathrm{Wh}(V(\prescript{w}{}{\omega^\prime}))$. \\
Such a functional is a Whittaker functional, and thus may be written as a linear combination
\[
    \langle \lambda_\eta, \mathrm{I}_w f\rangle = \sum_{\eta^\prime \in \widetilde{T}_\ast \backslash \widetilde{T}}\tau(w, \omega^\prime, \eta, \eta^\prime) \langle\lambda_{\eta^\prime}, f \rangle
\]
Suppose that the function $\mathbf{c}:\widetilde{T} \rightarrow \mathbb{C}$ satisfies $\mathbf{c}(h\eta) = (\omega^\prime \mu)(h^{-1})\mathbf{c}(\eta)$, for any $h \in \widetilde{T}_\ast, \eta \in \widetilde{T}$, 
and we define that 
\[
    \lambda(\mathbf{c}) = \sum_{\eta \in \widetilde{T}_\ast \backslash \widetilde{T}} \mathbf{c}(\eta) \lambda_\eta
\]
Since $\{\lambda_\eta\}_{\eta\in \widetilde{T}_\ast \backslash \widetilde{T}}$ is a basis, when $\omega$ is exceptional, swapping the summation, we see that 
$\lambda(\mathbf{c})$ is a Whittaker functional of $V_0(\omega^\prime)$ if and only if 
\[
\langle \lambda(\mathbf{c}), \mathrm{I}_s f\rangle=0 \Leftrightarrow \sum_{\eta \in \widetilde{T}_\ast \backslash \widetilde{T}} \mathbf{c}(\eta) \tau(s, \prescript{s}{}{\omega^\prime},\eta,\eta^\prime) = 0 
\]
for any simple reflection $s$, $\eta^\prime \in \widetilde{T}_\ast \backslash \widetilde{T}$ and $f\in V(\prescript{s}{}{\omega^\prime})$.
Our task is to determine $\tau(s, \prescript{s}{}{\omega^\prime},\eta,\eta^\prime)$. \\
Note that if we do not assume $\omega^\prime$ to be dominant, the integral $\mathrm{I}_w$ is understood as its regularization and $\lambda_\eta$
is defined in Bernstein's generic sense (see \cite[p.~75]{KP84}), thus the discussion above and the calculations afterwards still make sense.
\subsection{Statement of the Main Theorems} \label{main_thm}
In this section we keep the notation from the previous section. We give a rigorous statement of the main theorems.
Recall that 
\begin{itemize}
    \item $\mathbf{f}=(f_1,\cdots,f_r)$
    \item $\varpi^{\mathbf{f}}=\mathbf{s}(\mathrm{diag}\{\varpi^{f_1},\cdots,\varpi^{f_r}\})$
    \item canonical lifting $\kappa^\ast:\mathrm{GL}_r(O) \longrightarrow \widetilde{G}^{(c)}$ if $\lvert n \rvert=1$
\end{itemize}
\begin{theorem} \cite[Lemma~I.3.3]{KP84} \label{coefficients}
    Suppose $\lvert n \rvert=1$. Let $\omega$ be a genuine character on $\widetilde{T}_n \widetilde{Z}^{(c)}$ such that $\omega_\alpha^n$ is unramified for every simple root $\alpha=(i \enspace i+1), 1 \leq i \leq r-1$.
    Extend $\omega$ to $\widetilde{T}_\ast$, say $\omega^\prime$, such that it is trivial on $$\kappa^\ast(\mathrm{diag}\{1,\cdots,1,y,y^{-1},1,\cdots,1\}), \quad y \in O^\times$$
    where $y,y^{-1}$ lie at the $i$-th and $(i+1)$-th position respectively, for every $1 \leq i \leq r-1$.\\
    Let $$\theta = \omega^\prime(\kappa^\ast(\mathrm{diag}\{1,\cdots,1,-1,1,\cdots,1\}))$$ with $-1$ lying at the ($i+1$)-th position.\\
    Choose a set of representatives of $\widetilde{T}_\ast \backslash \widetilde{T}$ containing $\eta= \varpi^{\mathbf{f}}$ and let $\eta^\prime$ be one of the chosen representatives.
    Set
    \begin{itemize}
        \item $\tau_1(s,\omega^\prime,\eta,\eta^\prime) = \theta(1-\omega_\alpha^n(\varpi))^{-1}(1-q^{-1})\omega_\alpha^n(\varpi^{- \lfloor \frac{f_i-f_{i+1}}{n} \rfloor})$ only when $\eta\eta^{\prime-1}\in\widetilde{T}_\ast$, zero otherwise.
        \item $\tau_2(s,\omega^\prime,\varpi^{\mathbf f},\varpi^{s[\mathbf f]})=\theta \varepsilon(\delta)^{f_i f_{i+1}} q^{f_{i+1}-f_i-2} g^{(f_i-f_{i+1}+1),\overline{e}_0}$,
        and
        \[
        \tau_2(s,\omega^\prime,\varpi^{\mathbf f},\eta^\prime)=0
        \]
        when $\eta^\prime\notin\widetilde T_\ast\varpi^{s[\mathbf f]}$.
    \end{itemize}
    where
    \begin{itemize}
        \item $s[\mathbf{f}]:= ({f_1}, \cdots, f_{i-1},f_{i+1}-1,f_i+1,f_{i+2},\cdots,{f_r})$
        \item $\delta=\begin{cases}
        1   \quad \qquad ,n \text{ is odd} \\
        (-1)^{\frac{q-1}{n}}, n \text{ is even}
        \end{cases}$
        \item $g^{(I),e_0} = q \int_{O^\times} (\varpi,x)^I e_0(x/ \varpi) d^\times x, \enspace I \in \mathbb{Z}/n \mathbb{Z}$
    
    \end{itemize}
    Then we have \[\tau(s,\omega^\prime,\eta,\eta^\prime)=\tau_1(s,\omega^\prime,\eta,\eta^\prime) + \tau_2(s,\omega^\prime,\eta,\eta^\prime)\]
\end{theorem}
We postpone the lengthy proof to the next section. Now we can give a criterion of Whittaker functional. 
\begin{theorem} \cite[Corollary~I.3.4]{KP84} \label{formula}
    With the assumptions above, and we further suppose $\omega$ is exceptional.
    Then the functional $\lambda(\mathbf{c})$ defined above is a Whittaker functional of $V_0(\omega^\prime)$ if and only if, for any simple root $\alpha=(i,i+1)$, its associated simple reflection $s$, and every $\mathbf f=(f_1,\ldots,f_r)\in\mathbb Z^r$, one has
    \[
    \mathbf{c}(\varpi^{s[\mathbf{f}]})= 
        \varepsilon(\delta)^{(f_{i+1}-1)(f_i+1)} q^{f_{i+1}-f_i-\lfloor\frac{f_i-f_{i+1}}{n}\rfloor -2} g^{(f_i-f_{i+1}+1), e_0}\mathbf{c}(\varpi^{\mathbf{f}})
    \]
\end{theorem}
\begin{proof}
    Suppose $n\nmid f_i-f_{i+1}+1$ first.
    We have seen that being a Whittaker functional of $V_0(\omega^\prime)$ is equivalent to 
    \[
        \sum_{\eta \in \widetilde{T}_\ast \backslash \widetilde{T}} \mathbf{c}(\eta) \tau(s, \prescript{s}{}{\omega^\prime},\eta,\eta^\prime) = 0 
    \]
    Applying the previous theorem and letting $\eta^\prime = \varpi^{\mathbf{f}}$, this is equivalent to
    \[
        \tau(s, \prescript{s}{}{\omega^\prime}, \varpi^{s[\mathbf{f}]},\varpi^{\mathbf{f}})\mathbf{c}(\varpi^{s[\mathbf{f}]}) + \tau(s, \prescript{s}{}{\omega^\prime}, \varpi^{\mathbf{f}},\varpi^{\mathbf{f}})\mathbf{c}(\varpi^{\mathbf{f}})=0
    \]
    More explicitly,
    \[
        \theta \cdot \varepsilon(\delta)^{(f_{i+1}-1)(f_{i}+1)} q^{f_i+1-(f_{i+1}-1)-2} g^{(f_{i+1}-f_i-1),\overline{e}_0}\mathbf{c}(\varpi^{s[\mathbf{f}]})+ \theta (1-q)^{-1}(1-q^{-1})q^{-\lfloor\frac{f_i-f_{i+1}}{n}\rfloor}\mathbf{c}(\varpi^{\mathbf{f}})=0
    \]
    here we note that $\prescript{s}{}{\omega_\alpha^n}(x) = \omega_\alpha^n(x^{-1})$. 
    Combining $g^{(i),e_0} g^{(-i), \overline{e}_0} = q$ if $i \neq 0$, we deduce the equation desired. \\
    It remains to consider the case $n\mid f_i-f_{i+1}+1$. Put
    \[
    m=\frac{f_i-f_{i+1}+1}{n}.
    \]
    Moreover,
    \[
    g^{(f_i-f_{i+1}+1),\overline e_0}=-1,\qquad
    \varepsilon(\delta)^{f_if_{i+1}}=1,\qquad
    \left\lfloor\frac{f_i-f_{i+1}}{n}\right\rfloor=m-1.
    \]
    Applying Theorems~\ref{cocycle} and~\ref{coefficients} to $\prescript{s}{}{\omega^\prime}$ and using equivariance gives
    \[
    \begin{aligned}
    \tau_1(s,\prescript{s}{}{\omega^\prime},\varpi^{\mathbf f},\varpi^{\mathbf f})
    &=-\theta q^{-m},\\
    \tau_2(s,\prescript{s}{}{\omega^\prime},\varpi^{\mathbf f},\varpi^{\mathbf f})
    &=\theta\varepsilon(\delta)^{f_if_{i+1}}q^{(n-1)m+f_{i+1}-f_i-2}g^{(f_i-f_{i+1}+1),\overline e_0}\\
    &=-\theta q^{-m-1}.
    \end{aligned}
    \]
    Hence the Whittaker condition is
    \[
    -\theta q^{-m}(1+q^{-1})\mathbf c(\varpi^{\mathbf f})=0,
    \]
    which is equivalent to $\mathbf c(\varpi^{\mathbf f})=0$. Conversely, this vanishing makes the fixed-class equation hold. This completes the proof.
\end{proof}

\begin{remark}
    Actually our proof gives $\lambda(\mathbf{c}) \in \mathrm{Wh}(V_0(\omega^\prime))$ if and only if
    \[
        \mathbf{c}(\varpi^{s[\mathbf{f}]})= 
    \begin{cases}
        \varepsilon(\delta)^{(f_{i+1}-1)(f_i+1)} q^{f_{i+1}-f_i-\lfloor\frac{f_i-f_{i+1}}{n}\rfloor -2} g^{(f_i-f_{i+1}+1), e_0}\mathbf{c}(\varpi^{\mathbf{f}}), \quad n\nmid f_i-f_{i+1}+1\\
        0, \quad n\mid f_i-f_{i+1}+1
    \end{cases}
    \]
\end{remark}

\section{Calculation of the Coefficients} \label{calculation}
This section mainly follows \cite[Lemma~I.3.3]{KP84}, with some additional details. We use $\mathbf{I}$ for the identity matrix.
Recall that $\alpha = (i,i+1)$ with associated simple reflection $s$.\\
\subsection{Reduction to An Integral}
We first show that $\tau(s, \omega^\prime,\eta,\eta^\prime)$ can be expressed as an integral
by appropriately constructed test vectors.
By \cite[Proposition~0.1.2]{KP84}, let
\[
K_m = \{ k \in \mathrm{GL}_r(O) : k \equiv \mathbf{I} \pmod{\mathfrak{p}^m} \},
\]
and suppose that \(m \geq 1\) is large enough so that the covering \(p\) splits over \(K_m\). 
Let \(k \mapsto k^*\) be such a splitting, and let \(K_m^*\) be the image of \(K_m\) in \(\widetilde{G}\). 
Let \(K_m^+ = N_+ \cap K_m\). We can choose suitable extension $\omega^\prime$ of $\omega$ s.t. $\omega^\prime$ is trivial on $\kappa^\ast(\mathrm{diag}\{\mathbf{I},y,y^{-1},\mathbf{I}\}), y\in O^{\times}$, where $y,y^{-1}$ lie at $i$-th and $(i+1)$-th positions respectively.
Then we define \(f_\eta^{(m)}\) to be 
\begin{align*}
    f_\eta^{(m)}(b \eta w_0 k) &= (\omega^\prime \mu)(b), \quad b \in \widetilde{B}_\ast,\ k \in K_m^\ast, \\
    f_\eta^{(m)}(g) &= 0, \quad g \notin \widetilde{B}_\ast \eta w_0 K_m^\ast.
\end{align*}
Thus an element of \(V(\omega^\prime)\) satisfying
\begin{enumerate}
\item[a)] \(\mathrm{Supp}(f_\eta^{(m)}) = \widetilde{B}_* \eta w_0 K_m^* = \widetilde{B}_* \eta w_0 (K_m^+)^*\),
\item[b)] \(f_\eta^{(m)}\) is right-\(K_m^*\)-invariant
\item[c)] \(f_\eta^{(m)}(\eta w_0) = 1\).
\end{enumerate}

If \(m\) is sufficiently large, such an \(f_\eta^{(m)}\) exists. Applying the same trick as in \cite[Lemma~I.3.1]{KP84}, we have:
\[
\begin{aligned}
\langle \lambda_\eta, f_{\eta^\prime}^{(m)} \rangle &\neq 0 \quad (\eta = \eta^\prime) \\
&= 0 \quad (\eta \neq \eta^\prime).
\end{aligned}
\]
Then one has
\[
\langle \lambda_\eta, \mathrm{I}_s f_{\eta^\prime}^{(m)} \rangle = \tau(s, \omega^\prime, \eta, \eta^\prime) \langle \lambda_{\eta^\prime}, f_{\eta^\prime}^{(m)} \rangle,
\]
so that
\[
\tau(s, \omega^\prime,\eta, \eta^\prime) = \langle \lambda_\eta, \mathrm{I}_s f_{\eta^\prime}^{(m)} \rangle / \langle \lambda_{\eta^\prime}, f_{\eta^\prime}^{(m)} \rangle.
\]
Moreover in the proof of \cite[Lemma~I.3.1]{KP84} we saw that
\[
\langle \lambda_{\eta^\prime}, f_{\eta^\prime}^{(m)} \rangle = \mathrm{meas}\big((K_m^+)^*\big)
\]
if \(m\) is large enough. Consequently one has
\[
\tau(s, \omega^\prime, \eta, \eta^\prime) = \lim_{m \to \infty} \langle \lambda_\eta, \mathrm{I}_s f_{\eta^\prime}^{(m)} \rangle / \mathrm{meas}\big((K_m^+)^*\big).
\]
for sufficiently large \(m\). This then takes the form
\[
    \frac{1}{\mathrm{meas}((K_m^+)^*)} \int_{N_+^\ast} \left\{ \int_{N_+^\ast(s)} f_{\eta^\prime}^{(m)}(s n_1 \eta w_0^{-1} n_2) \, dn_1 \right\} \overline{e}(n_2) \, dn_2
\]
where the integrals are to be understood as their regularized values (see \cite[p.~66]{KP84}) since we do not assume that $\omega$ is dominant, and the integrals may not be convergent. 
\subsection{Block Matrix Decomposition}
Now we are faced with an integral on $r \times r$ matrices. We will simplify it to an integral on $F$ which is easier for computation by 
analyzing the integration area and block matrix decomposition.
In order to evaluate the inner integral we first determine when 
\(s n_1 \eta w_0^{-1} n_2\) can lie in \(\widetilde{B}_* \eta^\prime w_0 K_m^* = \widetilde{B}_* \eta^\prime w_0^{-1}(K_m^+)^*\). Suppose that
\[
s n_1 \eta w_0^{-1} n_2 = b \eta^\prime w_0^{-1} k \quad (b \in \widetilde{B}_*, k \in (K_m^+)^*),
\]
or
\[
p(s) p(n_1) = p(b) p(\eta^\prime) p(w_0^{-1} k n_2^{-1} w_0) p(\eta)^{-1}.
\]

We shall express matrices in $G$ in \((r_1 + 2 + r_2) \times (r_1 + 2 + r_2)\) block form so that \(p(s)\) is
\[
\begin{bmatrix}
\mathbf{I} & 0 & 0 \\
0 & \begin{pmatrix} 0 & 1 \\ 1 & 0 \end{pmatrix} & 0 \\
0 & 0 & \mathbf{I}
\end{bmatrix}.
\]

We can then write \(p(n_1)\) as
\[
\begin{bmatrix}
\mathbf{I} & 0 & 0 \\
0 & \begin{pmatrix} 1 & \xi \\ 0 & 1 \end{pmatrix} & 0 \\
0 & 0 & \mathbf{I}
\end{bmatrix}
\]
with \(\xi \in F\), \(p(b)\) as
\[
\begin{bmatrix}
b_1 & x_1 & x_2 \\
0 & b_2 & x_3 \\
0 & 0 & b_3
\end{bmatrix}
\]
with \(b_1, b_2, b_3\) upper-triangular, \(p(\eta)\) (resp. \(p(\eta^\prime)\)) as
\[
\begin{bmatrix}
\eta_1 & 0 & 0 \\
0 & \eta_2 & 0 \\
0 & 0 & \eta_3
\end{bmatrix}
\quad \left( \text{resp. }
\begin{bmatrix}
\eta_1' & 0 & 0 \\
0 & \eta_2' & 0 \\
0 & 0 & \eta_3'
\end{bmatrix}
\right)
\]
with \(\eta_j\) (resp. \(\eta_j'\)) diagonal, and \(p(w_0^{-1} k n_2^{-1} w_0)\) as
\[
\begin{bmatrix}
\bar{n}_1 & 0 & 0 \\
u_1 & \bar{n}_2 & 0 \\
u_2 & u_3 & \bar{n}_3
\end{bmatrix}.
\]
where \(\bar{n}_1, \bar{n}_2, \bar{n}_3\) are lower-triangular matrices. From
\[
p(s) p(n_1) = p(b) p(\eta^\prime) p(w_0^{-1} k n_2^{-1} w_0) p(\eta)^{-1}
\]
one obtains after a short computation that
\begin{align*}
u_1 &= 0, & u_2 &= 0, & u_3 &= 0, \\
x_1 &= 0, & x_2 &= 0, & x_3 &= 0, \\
\bar{n}_1 &= \mathbf{I}, & \bar{n}_3 &= \mathbf{I}, \\
b_1 &= \eta_1 \eta_1'^{-1}, & b_3 &= \eta_3 \eta_3'^{-1}
\end{align*}
and
\[
\begin{pmatrix} 0 & 1 \\ 1 & \xi \end{pmatrix} = b_2 \eta_2' \bar{n}_2 \eta_2^{-1}.
\]

One also demands that
\[
\begin{bmatrix}
\eta_1 \eta_1'^{-1} & 0 & 0 \\
0 & b_2 & 0 \\
0 & 0 & \eta_3 \eta_3'^{-1}
\end{bmatrix}
\]
should belong to \(\widetilde{B}_*\). From the last of the equations above
\[
b_2 = \begin{pmatrix} -\xi^{-1} & 1 \\ 0 & \xi \end{pmatrix} \eta_2 \eta_2^{\prime-1}
\]
and
\[
\bar{n}_2 = \begin{pmatrix} 1 & 0 \\ \nu \xi^{-1} & 1 \end{pmatrix}
\]
where
\[
\nu = p(\eta)^\alpha:= p(\eta)_{ii} / p(\eta)_{i+1,i+1}
\]
Let $z$ be the $(r-i,r-i+1)$-entry of $p(n_2)$. This forces the $2\times2$ block of $p(k)$ in rows and columns $r-i,r-i+1$ to be $\Bigl(\begin{smallmatrix}
    1 & \nu \xi^{-1}+z \\
    0 & 1
\end{smallmatrix}\Bigr)$.
Let now
\[
b(\xi) = \begin{bmatrix}
\mathbf{I} & 0 & 0 \\
0 & \begin{pmatrix} -\xi^{-1} & 1 \\ 0 & \xi \end{pmatrix} & 0 \\
0 & 0 & \mathbf{I}
\end{bmatrix} p(\eta) p(\eta^\prime)^{-1}
\]
and
\[
\widetilde{h}_\alpha(\xi) = \mathbf{s}\left(
\begin{bmatrix}
\mathbf{I} & 0 & 0 \\
0 & \begin{pmatrix} -\xi^{-1} & 0 \\ 0 & \xi \end{pmatrix} & 0 \\
0 & 0 & \mathbf{I}
\end{bmatrix}
\right).
\]

By our assumptions on $\omega^\prime$ and $\mu$, noting that $p(w_0^{-1}kn_2^{-1}w_0) = \mathrm{diag}\{\mathbf{I}, \bar{n}_2,\mathbf{I}\}$, the integral is then

\begin{align*}
&\frac{1}{\mathrm{meas}((K_m^+)^\ast)} \int_{N_+^\ast} \left\{ \int_{N_+^\ast(s)} f_{\eta^\prime}^{(m)}(s n_1 \eta w_0^{-1} n_2) \, dn_1 \right\} \overline{e}(n_2) \, dn_2 \\
&= \frac{1}{\mathrm{meas}((K_m^+)^\ast)} \int_{N_+^\ast / (K_m^+)^\ast} \int_{(K_m^+)^\ast} \int_{N_+^\ast(s)} f_{\eta^\prime}^{(m)}(s n_1 \eta w_0^{-1} u k) \overline{e}(u k) \, dn_1 \, dk \, du \\
&= \sum_{z \in F/\mathfrak{p}^m} \left\{ \int_{\Xi'(z, \eta, \eta^\prime)} (\omega^\prime \mu)(\widetilde{h}_\alpha(\xi) \eta {\eta^\prime}^{-1}) \, d\xi \right\} \overline{e}_0(z)
\end{align*}

where
\[
\Xi'(z, \eta, \eta^\prime) = \left\{ \xi \in F : \nu \xi^{-1} + z \in \mathfrak{p}^m,\ \widetilde{h}_\alpha(\xi) \eta \eta^{\prime -1} \in \widetilde{T}_* \right\}.
\]

Still, it should be understood as a regularized value, first of the integral inside the braces and then of the sum.
If \(z\) does not represent the zero class of \(F/\mathfrak{p}^m\) then apply change of variable $\nu \xi^{-1}+z = u \in \mathfrak{p}^m$ 
to the inner integral 
\[
\int_{\mathfrak{p}^m} (\omega^\prime \mu) (\widetilde{h}_\alpha(\frac{\nu}{u-z}) \eta \eta^{\prime -1}) |\nu| |z|^{-2} du
=(\omega^\prime \mu)\big(\widetilde{h}_\alpha(-\nu z^{-1}) \eta \eta^{\prime -1}\big) q^{-m} |z|^{-2} |\nu|,
\]
if \(\widetilde{h}_\alpha(-\nu z^{-1}) \eta \eta^{\prime -1} \in \widetilde{T}_*\) and is zero otherwise. Here we note that
$\widetilde{h}_\alpha(\frac{\nu}{u-z})$ and $\widetilde{h}_\alpha(-\nu z^{-1})$ differ by 
$\mathbf{s} (\mathrm{diag}\{\mathbf{I},\bigl(\begin{smallmatrix}
    y & 0 \\ 0 & y^{-1}
\end{smallmatrix}\bigr), \mathbf{I}\}) , \enspace y \in O^\times $, where $\omega^\prime \mu$ is trivial.

\subsection{Breaking Up the Integral}
We will show that the whole integral consists of two parts, which we will compute respectively. \\
The integral over \(\Xi'(0, \eta, \eta^\prime)\) is
\[
A(m) = \int \omega^\prime \mu\big(\widetilde{h}_\alpha(\xi) \eta \eta^{\prime -1}\big) d\xi
\]
integrated over those \(\xi \in F^\times\) such that \(\operatorname{ord}(\xi) \leq \operatorname{ord}(\nu)-m\) and \(\widetilde{h}_\alpha(\xi) \eta \eta^{\prime -1} \in \widetilde{T}_*\). 

Consider now the sum over \(z\). This sum we also arrange as the sum of the integral over 
\(\Xi'(0, \eta, \eta^\prime)\) and the sum of the subsums over \(|z| = q^j\) (\(j > -m\)). 
In these subsums we replace \(z\) by \(z(1+t)\) with \(t \in \mathfrak{p}\) (note that this does not affect the sum). Averaging over \(t\in\mathfrak{p}\), we obtain
\begin{align*}
&\sum_{z \in F/ \mathfrak{p}^m - \mathfrak{p}/\mathfrak{p}^m} (\omega^\prime \mu)\big(\widetilde{h}_\alpha(-\nu z^{-1}) \eta \eta^{\prime -1}\big) \overline{e}_0(z) q^{-m} |z|^{-2} |\nu|\\
=\;&\sum_{z \in F/ \mathfrak{p}^m - \mathfrak{p}/\mathfrak{p}^m} (\omega^\prime \mu)\big(\widetilde{h}_\alpha(-\nu z^{-1}) \eta \eta^{\prime -1}\big) \overline{e}_0(z) q^{-m} |z|^{-2} |\nu|
\frac{1}{\mathrm{meas}(\mathfrak{p})}\int_{\mathfrak{p}}\overline{e}_0(zt)\,dt.
\end{align*}
Since the conductor of \(e_0\) is \(O\),
\[
\frac{1}{\mathrm{meas}(\mathfrak{p})}\int_{\mathfrak{p}}\overline{e}_0(zt)\,dt
=\begin{cases}
1,& |z|\leq q,\\
0,& |z|>q.
\end{cases}
\]
Therefore the original summation with respect to $z$ vanishes unless \(|z| \leq q\).

The terms remaining are those with \(|z| = q\), which yield
\[
\sum_{\substack{z \in \mathfrak{p}^{-1} / O \\ z \neq 0}} (\omega^\prime \mu)\big(\widetilde{h}_\alpha(-\nu z^{-1}) \eta \eta^{\prime -1}\big) \overline{e}_0(z) q^{-2} |\nu|
\]
and those with \(|z| \leq 1\), namely
\[
B(m) = \sum_{\substack{z \in O / \mathfrak{p}^m \\ z \neq 0}} (\omega^\prime \mu)\big(\widetilde{h}_\alpha(-\nu z^{-1}) \eta \eta^{\prime -1}\big) q^{-m} |z|^{-2} |\nu|
\]

We emphasize that in these last two sums there is an additional restriction on \(z\) and the argument of \((\omega^\prime \mu)\) should lie in \(\widetilde{T}_*\).

Hence \(\tau(s, \omega^\prime, \eta, \eta^\prime)\) is the sum of
\[
\sum_{\substack{z \in \mathfrak{p}^{-1} / O \\ z \neq 0}} (\omega^\prime \mu)\big(\widetilde{h}_\alpha(-\nu z^{-1}) \eta \eta^{\prime -1}\big) \overline{e}_0(z) q^{-2} \lvert \nu \rvert
\]
and \(A(m) + B(m)\) for \(m\) sufficiently large. Thus \(A(m) + B(m)\) is constant, as a function of \(m\), for \(m\) sufficiently large.

\subsection{Simplifying via Recursion Property}
Using the asymptotic behaviour of $A(m)+B(m)$, we may express it via $B(m)$ and $\omega^n_\alpha$ by noting that there are recursion 
properties of $A(m)$ and $B(m)$. \\
If in the integral defining \(A(m + n)\) we replace \(\xi\) by \(\varpi^{-n} \xi_1\), we see that for all \(m\)
\begin{align*}
    A(m + n) =& \int_{\operatorname{ord}(\xi) \leq \operatorname{ord}(\nu)-m-n} (\omega^\prime \mu)(\widetilde{h}_\alpha(\xi)\eta \eta^{\prime -1})d\xi \\
             =& \int_{\operatorname{ord}(\xi_1) \leq \operatorname{ord}(\nu)-m} (\omega^\prime \mu)(\widetilde{h}_\alpha(\varpi^{-n}\xi_1)\eta \eta^{\prime -1}) |\varpi^{-n}|d\xi_1 \\
             =& \omega_\alpha^n(\varpi)  \int_{\operatorname{ord}(\xi_1) \leq \operatorname{ord}(\nu)-m} (\omega^\prime \mu)(\widetilde{h}_\alpha(\xi_1)\eta \eta^{\prime -1}) \mu(\mathrm{diag} \{\mathbf{I},\Bigl(
             \begin{matrix}
                \varpi^n & 0 \\
                0     & \varpi^{-n}
             \end{matrix}
             \Bigr),\mathbf{I}\}) |\varpi^{-n}|d\xi_1 \\
             =&\omega_\alpha^n(\varpi) A(m).
\end{align*}
Clearly we have $A(m) = (\omega_\alpha^n(\varpi))^{m/n-1} A(n)$ if \(m \geq 0\), \(m \equiv 0 \pmod{n}\). 
Moreover one has
\begin{align*}
    B(m + n) =& \sum_{\substack{z \in O / \mathfrak{p}^{m+n} \\ z \neq 0}} (\omega^\prime \mu)(\widetilde{h}_\alpha(-\nu z^{-1}) \eta \eta^{\prime -1})q^{-m-n} |z|^{-2} |\nu| \\
             =& \int_{O - \mathfrak{p}^{m+n}} (\omega^\prime \mu) (\widetilde{h}_\alpha(-\nu z^{-1}) \eta \eta^{\prime -1}) |z|^{-2} |\nu| dz\\
             =& \int_{O - \mathfrak{p}^m} (\omega^\prime \mu)(\widetilde{h}_\alpha(-\nu z^{-1}) \eta \eta^{\prime -1}) |z|^{-2} |\nu| dz + \int_{\mathfrak{p}^m - \mathfrak{p}^{m+n}} (\omega^\prime \mu)(\widetilde{h}_\alpha(-\nu z^{-1}) \eta \eta^{\prime -1}) |z|^{-2} |\nu| dz \\
\end{align*}
In the second term here we replace \(z\) by \(\varpi^m z_1\) and obtain
\begin{align*}
    B(m + n) =& B(m) + \int_{O - \mathfrak{p}^{n}} (\omega^\prime \mu) (\widetilde{h}_\alpha(-\nu \varpi^{-m} z_1^{-1}) \eta \eta^{\prime -1}) |z_1|^{-2} q^m |\nu| dz_1\\
             =& B(m) + \omega_\alpha^n(\varpi^{m/n}) \int_{O - \mathfrak{p}^{n}} (\omega^\prime \mu) (\widetilde{h}_\alpha(-\nu z_1^{-1}) \eta \eta^{\prime -1}) |z_1|^{-2} 
             \mu 
             (\mathrm{diag}\{\mathbf{I},\Bigl(
                \begin{matrix}
                    \varpi^m & 0 \\
                    0     & \varpi^{-m}
                \end{matrix}\Bigr), \mathbf{I}\} )
                q^m |\nu| dz_1 \\
             =& B(m) + \omega_\alpha^n(\varpi^{m/n}) B(n)
\end{align*}
Applying summation over $m=kn$, we have
\[
    B(m) = B(n) \frac{1-(\omega_\alpha^n(\varpi))^{m/n}}{1- \omega_\alpha^n(\varpi)}
\]
here $m \geq 0$, \(m \equiv 0 \pmod{n}\). \\
From these it follows firstly that for large \(m\), \(m \equiv 0 \pmod{n}\), one has
\[
    A(m+n)+B(m+n) = \omega_\alpha^n(\varpi) A(m) + B(m) + \omega_\alpha^n(\varpi^{m/n}) B(n) = A(m) + B(m)
\]
Thus one has, for some large $m$,
\[
    (\omega_\alpha^n(\varpi))^{m/n-1} A(n)=A(m) = \omega_\alpha^n(\varpi^{m/n}) \left( 1 - \omega_\alpha^n(\varpi) \right)^{-1} B(n)
\]
and so
\[
    A(n) = (1- \omega_\alpha^n(\varpi))^{-1} \omega_\alpha^n(\varpi)B(n)
\]
At last, combining the equations above, for any $m \geq 0$, $m \equiv 0 \pmod{n}$, we deduce
\[
\begin{aligned}
    A(m) + B(m) &= (1- \omega_\alpha^n(\varpi))^{-1} B(n) \\
\end{aligned}
\]

\subsection{Final Calculations}
Therefore \(\tau(s, \omega^\prime, \eta, \eta^\prime)\) is the sum of
\[
    \sum_{\substack{z \in \mathfrak{p}^{-1} / O \\ z \neq 0}} (\omega^\prime \mu)\big(\widetilde{h}_\alpha(-\nu z^{-1}) \eta \eta^{\prime -1}\big) \overline{e}_0(z) q^{-2} |\nu|
\]
and
\begin{equation*}
    \left( 1 - \omega_\alpha^n(\varpi) \right)^{-1} \sum_{\substack{z \in O / \mathfrak{p}^n \\ z \neq 0}} (\omega^\prime \mu)\big(\widetilde{h}_\alpha(-\nu z^{-1}) \eta \eta^{\prime -1}\big) q^{-n} |z|^{-2} |\nu| \tag*{$(\ast)$}
\end{equation*}

Now we calculate them respectively. In the first term set \(z = y/\varpi\). 
Take \(\eta\) to be of the form \(\varpi^{\mathbf{f}}\). Under the assumption made on \(\omega^\prime\) the sum becomes
\[
    \sum_{\substack{y \in O / \mathfrak{p} \\ y \neq 0}} (\omega^\prime \mu)\big(\widetilde{h}_\alpha(-\nu \varpi) \eta \eta^{\prime -1}\big) \varepsilon(\nu \varpi, y) \overline{e}_0(y/\varpi) q^{-2} |\nu|
\]

Clearly the class of \(\eta^\prime\) in \(\widetilde{T}_\ast \backslash \widetilde{T}\) is determined and we see that it is the class of \(\varpi^{s[\mathbf{f}]}\) by the description of the equivalence classes. Thus $\tau_2$ vanishes outside this class. To calculate its value at $\varpi^{s[\mathbf f]}$, take $\eta^\prime=\varpi^{s[\mathbf f]}$.
If now
\[
    J = f_{i+1} - f_i - 1
\]
then the sum is
\[
    q^{-2} (\omega^\prime \mu)\big(\widetilde{h}_\alpha(-\nu \varpi) \eta \eta^{\prime -1}\big) g^{(-J), \overline{e}_0} \lvert \nu \rvert,
\]
here
\[
    g^{(I),e_0} = q \int_{O^\times} (\varpi,x)^I e_0(x/ \varpi) d^\times x, \enspace I \in \mathbb{Z}/n \mathbb{Z}
\]
Using the formula in Theorem~\ref{cocycle}, and after complicated computations at the end of this section, we have
\[
    (\omega^\prime \mu) (\widetilde{h}_\alpha(-\nu \varpi) \eta \eta^{\prime -1}) = (\omega^\prime \mu)\Bigl(
    \mathrm{diag}\{\mathbf{I},\Bigl(\begin{matrix}
        1 & 0 \\
        0 & -1
    \end{matrix}\Bigr),\mathbf{I}\}, (\varpi,\varpi)^{f_i f_{i+1}-f_{i+1}^2+f_{i+1}}\Bigr) 
\]
which is just
\[
    \theta \cdot \varepsilon((\varpi,\varpi))^{f_i f_{i+1}-f_{i+1}^2+f_{i+1}}
\]
where $\theta = \omega^\prime(\kappa^\ast(\mathrm{diag}\{1,\cdots,1,-1,1,\cdots,1\}))$ with $-1$ lying at the ($i+1$)-th position. 
Note that we assumed $\lvert n \rvert= 1$, so using the formula \cite[Chapter~V, Proposition~3.4]{neukirch2013algebraic} for the tame Hilbert symbol, we may simplify $(\varpi,\varpi)$.\\
Now consider the second term. It is invariant under replacing \(z\) by \(zu\) (\(u \in O^\times\)), and we apply the same trick by summing over $u \in O^\times/(1+ \mathfrak{p})$, that is
\[
    (q-1) \times (\ast) =  \left( 1 - \omega_\alpha^n(\varpi) \right)^{-1} \sum_{\substack{z \in O / \mathfrak{p}^n \\ z \neq 0}} (\omega^\prime \mu)\big(\widetilde{h}_\alpha(-\nu z^{-1}) \eta \eta^{\prime -1}\big) q^{-n} |z|^{-2} |\nu| \Bigl(\sum_{u \in O^\times/(1+ \mathfrak{p})} (\nu z^{-1}, u)\Bigr)
\]
and from this one sees that the only terms which contribute are those with 
\[(\nu z^{-1}, -) \enspace \textrm{trivial} \Leftrightarrow \operatorname{ord}(\nu z^{-1}) \equiv 0 \pmod{n}\]
We may assume that $0 \leq \operatorname{ord}(z) < n$. Since $\nu = \varpi^{f_i - f_{i+1}}$, we may choose \[z = \varpi^{f_i - f_{i+1}-n \lfloor \frac{f_i - f_{i+1}}{n} \rfloor}\]
Recall that the sum vanishes unless $\widetilde{h}_\alpha(-\nu z^{-1})\eta \eta^{\prime -1} \in \widetilde{T}_\ast$, and we may choose $\eta^\prime = \eta$ to meet this. 
Also, this is the only choice if the sum does not vanish.
Thus we can restrict our attention to the case \(\eta^\prime = \eta\), and here, letting $k$ be the $\varpi$-adic valuation of our chosen $z$, we obtain
\begin{align*}
    &\left( 1 - \omega_\alpha^n(\varpi) \right)^{-1} \int_{|z'| = q^{-k}} (\omega^\prime \mu)\big(\widetilde{h}_\alpha(-\nu z^{\prime -1}) \big) q^{2k} |\nu| dz' \\
    =&\left( 1 - \omega_\alpha^n(\varpi) \right)^{-1} (q^{-k}-q^{-(k+1)}) (\omega^\prime \mu)\big(\widetilde{h}_\alpha(-\nu z^{-1}) \big) q^{2k} |\nu| \\
    =&\theta\left( 1 - \omega_\alpha^n(\varpi) \right)^{-1} (q^{-k}-q^{-(k+1)}) \omega_\alpha^n(\varpi^{- \lfloor \frac{f_i - f_{i+1}}{n} \rfloor})  \\ &\times \mu(\mathrm{diag}\{\mathbf{I},\Bigl(
        \begin{smallmatrix}
            \varpi^{-n \lfloor \frac{f_i - f_{i+1}}{n} \rfloor} & 0 \\
            0 & \varpi^{n \lfloor \frac{f_i - f_{i+1}}{n} \rfloor}
        \end{smallmatrix}
        \Bigr),\mathbf{I}\})q^{2k} |\nu| \\
    =& \theta \left( 1 - \omega_\alpha^n(\varpi) \right)^{-1}(1-q^{-1}) \omega_\alpha^n(\varpi^{-\lfloor \frac{f_i - f_{i+1}}{n} \rfloor})
\end{align*}
Thus our calculation completes.\\
\subsection{Cocycle Calculations}
For the completeness of the proof, we calculate the second factor above.
Recall that 
\begin{itemize}
    \item $\eta= \mathbf{s}(\mathrm{diag}\{\varpi^{f_1},\cdots,\varpi^{f_r}\})$
    \item $\eta^\prime=\mathbf{s}(\mathrm{diag}\{\varpi^{f_1},\cdots,\varpi^{f_{i-1}},\varpi^{f_{i+1}-1},\varpi^{f_{i}+1},\varpi^{f_{i+2}},\cdots,\varpi^{f_r}\})$
    \item $\nu= \varpi^{f_i-f_{i+1}}$
\end{itemize}
We calculate $\eta^{\prime -1}$ first. Put
\[
t=\mathrm{diag}\{\varpi^{f_1},\ldots,\varpi^{f_r}\},
\qquad
t'=\mathrm{diag}\{\varpi^{f_1},\ldots,\varpi^{f_{i+1}-1},
\varpi^{f_i+1},\ldots,\varpi^{f_r}\}.
\]
Then
\[
\eta=(t,1),\qquad
\eta'=(t',1),\qquad
\eta^{\prime -1}=(t^{\prime -1},x),
\qquad
x=\sigma(t',t^{\prime -1})^{-1}.
\]
Therefore
\begin{align*}
\eta\eta^{\prime -1}
&=(tt^{\prime -1},A),\\
A
&=x\,\sigma(t,t^{\prime -1})
=\frac{\sigma(t,t^{\prime -1})}
{\sigma(t',t^{\prime -1})}.
\end{align*}
The restriction of the cocycle to the diagonal torus is bimultiplicative. Hence
\[
A=\sigma(t't^{-1},t').
\]
Set
\[
J=f_{i+1}-f_i-1.
\]
Then
\[
t't^{-1}
=\mathrm{diag}\{\mathbf{I},\varpi^J,\varpi^{-J},\mathbf{I}\}.
\]
Applying Theorem~\ref{cocycle}, the factors involving positions outside $i$ and $i+1$ cancel, and we obtain
\begin{align*}
A
={}&(\varpi^J,\varpi^{f_{i+1}-1})^c
(\varpi^{-J},\varpi^{f_i+1})^c\\
&\times
(\varpi^J,\varpi^{f_i+1})^{c+1}
(\varpi^{-J},\varpi^{f_{i+1}-1})^c\\
={}&(\varpi^J,\varpi^{f_i+1})\\
={}&(\varpi^{f_{i+1}-f_i-1},\varpi^{f_i+1}).
\end{align*}
Thus
\begin{align*}
\widetilde{h}_\alpha (-\nu \varpi) \eta \eta^{\prime -1}
={}&\Bigl(\mathrm{diag}\{\mathbf{I},\nu^{-1}\varpi^{-1},-\nu \varpi,\mathbf{I}\},1\Bigr)\\
&\times
\Bigl(\mathrm{diag}\{\mathbf{I},\varpi^{-J},\varpi^J,\mathbf{I}\},A\Bigr).
\end{align*}
Now return to $\widetilde{h}_\alpha (-\nu \varpi) \eta \eta^{\prime -1}$, use Theorem~\ref{cocycle} again, we have
\begin{align*}
    \widetilde{h}_\alpha (-\nu \varpi) \eta \eta^{\prime -1}=
    &\Bigl(\mathrm{diag}\{\mathbf{I},1,-1,\mathbf{I}\},\sigma \Big(\Bigl(\begin{smallmatrix}
        \nu^{-1}\varpi^{-1} & 0\\
        0 & -\nu \varpi
    \end{smallmatrix}\Bigr),
    \Bigl(\begin{smallmatrix}
        \varpi^{1-f_{i+1}+f_i} & 0\\
        0 & \varpi^{-1+f_{i+1}-f_i}
    \end{smallmatrix}\Bigr)
    \Big) \cdot A\Bigr) \\
    =& \Bigl(\mathrm{diag}\{\mathbf{I},1,-1,\mathbf{I}\},
    (\varpi,\varpi)^{J^2}A\Bigr)\\
    =& \Bigl(\mathrm{diag}\{\mathbf{I},1,-1,\mathbf{I}\},
    (\varpi,\varpi)^{f_i f_{i+1}-f_{i+1}^2+f_{i+1}}\Bigr)
\end{align*}
where in the last equality we use $(\varpi,\varpi)^2=1$.
This completes the calculation.

\section{Comparison with the Formula in [Gao17]} \label{compare}
We compare \cite[Corollary~3.7]{GaoFan2017Dtrf} with Theorem~\ref{formula}, assuming throughout that $F$ has characteristic zero, $\lvert n\rvert=1$, and $\omega^\prime$ is the canonical unramified extension of its central restriction. Thus, in addition to the assumptions of Theorem~\ref{formula}, we require $\omega^\prime(\kappa^\ast(t))=1$ for $t\in T\cap K$, where $K=\mathrm{GL}_r(O)$.

Let $Y=\mathbb Z^r$ be the cocharacter lattice with standard basis $e_1,\ldots,e_r$ and simple coroots $\alpha_i^\vee=e_i-e_{i+1}$. For the twist parameter $0\leq c<n$, choose the bisector and quadratic form
\[
D(e_i,e_j)=\begin{cases}c+1,&i<j,\\c,&i\geq j,\end{cases}
\qquad B_Q=D+D^{\mathsf T},\qquad Q(y)=D(y,y),
\]
and take $\eta$ to be trivial on the coroot lattice. Then $Q(\alpha_i^\vee)=-1$ and $n_{\alpha_i}=n$. By \cite[Example~3.16]{GaoFan2017Dtrf}, these data give the Kazhdan--Patterson cover $\widetilde G$; the torus cocycle agrees with Theorem~\ref{cocycle}.

Choose the simple root pinnings
\[
e_{\alpha_i}(x)=\mathbf I-xE_{i,i+1},\qquad
e_{-\alpha_i}(x)=\mathbf I-xE_{i+1,i},
\]
where $E_{a,b}$ is the standard matrix unit. By the rank-two cocycle in \cite[p.~41]{KP84} and block compatibility, 
the natural coroot lift is $\overline h_{\alpha_i}(a)=\mathbf s(\alpha_i^\vee(a))$, 
consistent with \cite[equation~(2), p.~340]{GaoFan2017Dtrf}.
For $1\leq i<r$, let $s_i=(i,i+1)$ and put
\[
s_i^{\prime}=e_{\alpha_i}(1)e_{-\alpha_i}(-1)e_{\alpha_i}(1).
\]
Fix a reduced expression $w_0=s_{i_1}\cdots s_{i_\ell}$ and define
\[
w_0^{\prime}=s_{i_1}^{\prime}\cdots s_{i_\ell}^{\prime},\qquad
\widetilde w_0=\kappa^\ast(w_0^{\prime}).
\]
Multiplying these representatives and using the compact splitting \cite[Proposition~0.1.3(i),(ii), p.~44]{KP84} gives
\[
(\widetilde w_0)^{-1}=\widetilde d_0w_0^{-1},\qquad
\widetilde d_0=\kappa^\ast(d_0),\qquad
d_0=\operatorname{diag}\{1,-1,\ldots,(-1)^{r-1}\}.
\]
Taking Gao's scalar additive character to be $\psi=e_0^{-1}$ then gives $\psi_U=e$ by \cite[\S~3A, p.~344]{GaoFan2017Dtrf}.

Put $\epsilon=\varepsilon((-1,\varpi))$ and write $[t,u]=tut^{-1}u^{-1}$. The torus commutator formula gives
\[
\mathbf v=\sum_{j=1}^r(j-1)e_j,\qquad
\varepsilon([\varpi^{\mathbf f},\widetilde d_0])
=\epsilon^{B_Q(\mathbf f,\mathbf v)}.
\]
Since $\widetilde d_0\in\widetilde T_\ast$ has order dividing two, $[t,\widetilde d_0]^2=1$ for all $t\in\widetilde T$, and $[t,\widetilde d_0]=1$ for $t\in\widetilde T_\ast$.

Under induction in stages, put $\Phi_f(g)(t)=\mu(p(t))^{-1}f(tg)$. Following \cite[\S~3A, pp.~343--344]{GaoFan2017Dtrf}, Gao's Whittaker functional is defined by
\[
\langle\lambda_\eta^{\mathrm{Gao}},f\rangle:=\left(\int_{N_+^\ast}\Phi_f(\widetilde w_0^{-1}n)\overline{e(n)}\,dn\right)(\eta^{-1}),\qquad f\in V(\omega^\prime),\ \eta\in\widetilde T.
\]
His coefficient function $C:\widetilde T\to\mathbb C$ satisfies
\[
C(th)=C(t)\omega^\prime(h)\qquad(t\in\widetilde T,\ h\in\widetilde T_\ast)
\]
and parametrizes the Whittaker functional
\[
\lambda^{\mathrm{Gao}}(C):=\sum_{\eta\in\widetilde T/\widetilde T_\ast}C(\eta)\lambda_\eta^{\mathrm{Gao}}.
\]
Using $\widetilde w_0^{-1}=\widetilde d_0w_0^{-1}$, we obtain
\[
\begin{aligned}
\langle\lambda_\eta^{\mathrm{Gao}},f\rangle
&=\mu(p(\eta))\int_{N_+^\ast}\overline{e(n)}
f(\eta^{-1}\widetilde d_0w_0^{-1}n)\,dn\\
&=\mu(p(\eta))\varepsilon([\eta^{-1},\widetilde d_0])
\langle\lambda_{\eta^{-1}},f\rangle.
\end{aligned}
\]
Here we move $\widetilde d_0$ to the left and use $\omega^\prime(\widetilde d_0)=\mu(d_0)=1$. Comparing coefficients in $\lambda^{\mathrm{Gao}}(C)=\lambda(\mathbf c)$ therefore gives
\[
C(t^{-1})=\mu(p(t))\varepsilon([t,\widetilde d_0])^{-1}\mathbf c(t).
\]
This relation preserves equivariance.
For the chosen uniformizer sections, the central factor $(\varpi^{-\mathbf f})^{-1}=\zeta_{\mathbf f}\varpi^{\mathbf f}$, with $\zeta_{\mathbf f}\in\mu_n$, therefore gives
\[
    C(\varpi^{-\mathbf f})
    =\mu(\varpi^{\mathbf f})\epsilon^{-B_Q(\mathbf f,\mathbf v)}\varepsilon(\zeta_{\mathbf f})^{-1}
    \mathbf c(\varpi^{\mathbf f}).
\]

Let $\rho$ be half the sum of the positive coroots and $y_\rho=y-\rho$. For a simple root $\alpha$, Gao's notation specializes to
\[
\begin{aligned}
w_\alpha[y]&=w_\alpha(y-\rho)+\rho=w_\alpha(y)+\alpha^\vee,
&k_{y,\alpha}&=\left\lceil\frac{\langle y,\alpha\rangle}{n}\right\rceil,\\
\Gamma(y,\alpha^\vee)&=\epsilon^{\langle y_\rho,\alpha\rangle D(y,\alpha^\vee)},
&g_{e_0}(I)&=q^{-1}g^{(-I),e_0}.
\end{aligned}
\]
By \cite[Corollary~3.7]{GaoFan2017Dtrf}, we have $\lambda^{\mathrm{Gao}}(C)\in\mathrm{Wh}(V_0(\omega^\prime))$ if and only if
\begin{equation}\label{eq:gao_formula}
C(\varpi^{w_\alpha[y]})=q^{k_{y,\alpha}-1}\Gamma(y,\alpha^\vee)
\left[g_{e_0}(\langle y_\rho,\alpha\rangle Q(\alpha^\vee))\right]^{-1}C(\varpi^y)
\end{equation}
for every simple root $\alpha$ and $y\in Y$. The basis conversion identifies $y=-\mathbf f$. We now compare this formula with Theorem~\ref{formula} for $\alpha=(i,i+1)$.

\textbf{1. The Inner Product and Weyl Action:} \\
The standard pairing evaluates to $\langle y, \alpha \rangle = \langle -\mathbf{f}, \alpha \rangle = f_{i+1} - f_i$. 
Consequently, Gao's shifted inner product is $\langle y_\rho, \alpha \rangle = f_{i+1} - f_i - 1$. 
Applying Gao's shifted Weyl action to our inverted coordinate gives $w_\alpha[-\mathbf{f}] = w_\alpha(-\mathbf{f}) + \alpha^\vee = -(f_1, \dots, f_{i+1}-1, f_i+1, \dots, f_r)$. This is exactly $-s[\mathbf{f}]$. 

\textbf{2. The $q$-Exponent:} \\
We have $k_{y,\alpha} = \lceil \frac{f_{i+1} - f_i}{n} \rceil$. Using the identity $\lceil X \rceil = -\lfloor -X \rfloor$, this bound becomes $-\lfloor \frac{f_i - f_{i+1}}{n} \rfloor$. Gao's base exponent $q^{k_{y,\alpha}-1}$ thus translates to $q^{-\lfloor \frac{f_i - f_{i+1}}{n} \rfloor - 1}$.

Next, we substitute the relation between $C$ and $\mathbf c$ into \eqref{eq:gao_formula}. This multiplies the right-hand side by
\[
\frac{\mu(\varpi^{\mathbf f})}{\mu(\varpi^{s[\mathbf f]})}
\epsilon^{B_Q(s[\mathbf f]-\mathbf f,\mathbf v)}
\frac{\varepsilon(\zeta_{s[\mathbf f]})}{\varepsilon(\zeta_{\mathbf f})}
\]
with $C$ replaced with $\mathbf{c}$. The power of $\epsilon$ is a sign, whose cancellation with the Gauss-sum sign is checked next.
For $\mathrm{GL}_r$, the modular character on a diagonal element is $\mu(\varpi^{\mathbf{f}}) = q^{-\sum_{j=1}^r \frac{r+1-2j}{2} f_j}$. 
Direct calculations show that the exponent of $q$ should be
\[
    \left(-\lfloor \frac{f_i - f_{i+1}}{n} \rfloor - 1\right) + (f_{i+1} - f_i - 1) = f_{i+1} - f_i - \lfloor \frac{f_i - f_{i+1}}{n} \rfloor - 2
\]
This identically matches the $q$-exponent in Theorem \ref{formula}.

\textbf{3. The Gauss Sum:} \\
Since $\psi=e_0^{-1}$, Gao's factor $g_{\psi^{-1}}$ is $g_{e_0}$. The change of variable $x\mapsto-x$ gives
\[
g^{(f_i-f_{i+1}+1),\overline e_0}=\epsilon^{f_i-f_{i+1}+1} g^{(f_i-f_{i+1}+1),e_0}.
\]
Together with $g^{(i),e_0}g^{(-i),\overline e_0}=q$ for $n\nmid i$, this yields
\[
    \left[g_{e_0}\bigl(\langle y_\rho,\alpha\rangle Q(\alpha^\vee)\bigr)\right]^{-1}
    =\begin{cases}
    \epsilon^{f_i-f_{i+1}+1} g^{(f_i-f_{i+1}+1),e_0},&n\nmid(f_i-f_{i+1}+1),\\
    -q,&n\mid(f_i-f_{i+1}+1).
    \end{cases}
\]
On the other hand, our bilinear form satisfies
\[
B_Q(x,z)=(2c+1)\Bigl(\sum_jx_j\Bigr)\Bigl(\sum_jz_j\Bigr)-\sum_jx_jz_j,
\qquad B_Q(\alpha_i^\vee,\mathbf v)=1.
\]
As $s[\mathbf f]-\mathbf f=(f_{i+1}-f_i-1)\alpha_i^\vee$, we obtain
\[
\epsilon^{B_Q(s[\mathbf f]-\mathbf f,\mathbf v)}=\epsilon^{f_{i+1}-f_i-1}.
\]
Thus this factor cancels the sign $\epsilon^{f_i-f_{i+1}+1}$ when $n\nmid(f_i-f_{i+1}+1)$. If $n\mid(f_i-f_{i+1}+1)$, both signs equal $1$. Consequently,
\[
\frac{\epsilon^{B_Q(s[\mathbf f]-\mathbf f,\mathbf v)}}{g_{e_0}\bigl(\langle y_\rho,\alpha\rangle Q(\alpha^\vee)\bigr)}
=\begin{cases}
g^{(f_i-f_{i+1}+1),e_0},&n\nmid(f_i-f_{i+1}+1),\\
q\,g^{(f_i-f_{i+1}+1),e_0},&n\mid(f_i-f_{i+1}+1).
\end{cases}
\]

\textbf{4. The Factor $\Gamma(y,\alpha^\vee)$:} \\
To evaluate $\Gamma(y, \alpha^\vee)$, we use the bisector $D$ parameterized previously. For $y = -\mathbf{f}$, we have $D(-\mathbf{f}, \alpha^\vee) = f_i$. 
Then the exponent of $\epsilon$ becomes:
\[
    \langle y_\rho, \alpha \rangle D(-\mathbf{f}, \alpha^\vee)
    =f_i(f_{i+1}-f_i-1)
    =f_if_{i+1}-f_i(f_i+1)
\]
Since $f_i(f_i+1)$ is even and $\epsilon=\pm 1$, it reduces to $\epsilon^{f_i f_{i+1}}$. 
If we further assume $\lvert n\rvert=1$, using the formula for the tame Hilbert symbol \cite[Chapter~V, Proposition~3.4]{neukirch2013algebraic}, we have $\epsilon=\varepsilon(\delta)$.
\[
    \frac{\varepsilon(\zeta_{s[\mathbf f]})}{\varepsilon(\zeta_{\mathbf f})}
    =\varepsilon((\varpi,\varpi))^{f_{i+1}-f_i-1}
    =\varepsilon(\delta)^{f_{i+1}-f_i-1},
\]
\[
    \Gamma(-\mathbf f,\alpha^\vee)
    \frac{\varepsilon(\zeta_{s[\mathbf f]})}{\varepsilon(\zeta_{\mathbf f})}
    =\varepsilon(\delta)^{(f_{i+1}-1)(f_i+1)}.
\]
Combining these identities, we recover the formula in Theorem~\ref{formula} when $n\nmid f_i-f_{i+1}+1$.
If $n\mid f_i-f_{i+1}+1$, put $m=(f_i-f_{i+1}+1)/n$. Equivariance and the translated Gao relation give, respectively,
\[
\mathbf c(\varpi^{s[\mathbf f]})
=q^{-(n+1)m}\mathbf c(\varpi^{\mathbf f}),
\qquad
\mathbf c(\varpi^{s[\mathbf f]})
=-q^{1-(n+1)m}\mathbf c(\varpi^{\mathbf f}).
\]
Hence $\mathbf c(\varpi^{\mathbf f})=0$, the same vanishing condition obtained in the proof of Theorem~\ref{formula}. Under the assumptions of this section, Gao's criterion is therefore equivalent to that of Theorem~\ref{formula}.
Thus we see that under the framework of Brylinski-Deligne covering groups, the cocycle calculations are greatly simplified.

\printbibliography

@article{KP84,
  author    = {Kazhdan, David A. and Patterson, S. J.},
  title     = {Metaplectic forms},
  journal   = {Publications Math\'ematiques de l'IH\'ES},
  volume    = {59},
  pages     = {35--142},
  year      = {1984},
  publisher = {Institut des Hautes \'Etudes Scientifiques},
  doi       = {10.1007/BF02698770},
  mrnumber  = {743816},
  zbl       = {0559.10026},
  url       = {https://www.numdam.org/articles/10.1007/BF02698770/}
}

@book{neukirch2013algebraic,
  author     = {Neukirch, J{\"u}rgen},
  translator = {Schappacher, Norbert},
  title      = {Algebraic Number Theory},
  series     = {Grundlehren der mathematischen Wissenschaften},
  number     = {322},
  publisher  = {Springer},
  location   = {Berlin},
  year       = {1999},
  isbn       = {978-3-540-65399-8},
  doi        = {10.1007/978-3-662-03983-0},
  url        = {https://link.springer.com/book/10.1007/978-3-662-03983-0}
}

@article{Banks1999BlockcompatibleMC,
  author  = {Banks, William D. and Levy, Jason and Sepanski, Mark R.},
  title   = {Block-compatible metaplectic cocycles},
  journal = {Journal f{\"u}r die reine und angewandte Mathematik},
  volume  = {507},
  pages   = {131--163},
  year    = {1999},
  doi     = {10.1515/crll.1999.507.131},
  url     = {https://doi.org/10.1515/crll.1999.507.131}
}

@misc{cai2025coperiodscentralsymmetriccube,
  author        = {Cai, Li and Fan, Yangyu and She, Dongming},
  title         = {Co-periods and central symmetric cube {L}-values},
  year          = {2025},
  eprint        = {2507.15279},
  archiveprefix = {arXiv},
  primaryclass  = {math.NT},
  url           = {https://arxiv.org/abs/2507.15279}
}

@book{Kub69,
  author    = {Kubota, Tomio},
  title     = {On Automorphic Functions and the Reciprocity Law in a Number Field},
  series    = {Lectures in Mathematics},
  number    = {2},
  publisher = {Kinokuniya},
  location  = {Tokyo},
  year      = {1969},
  pagetotal = {65},
  url       = {https://ci.nii.ac.jp/ncid/BA06711865}
}

@article{BD01,
  author    = {Brylinski, Jean-Luc and Deligne, Pierre},
  title     = {Central extensions of reductive groups by {$K_2$}},
  journal   = {Publications Math\'ematiques de l'IH\'ES},
  volume    = {94},
  pages     = {5--85},
  year      = {2001},
  publisher = {Institut des Hautes \'Etudes Scientifiques},
  doi       = {10.1007/s10240-001-8192-2},
  mrnumber  = {1896177},
  zbl       = {1093.20027},
  url       = {https://www.numdam.org/articles/10.1007/s10240-001-8192-2/}
}

@article{GaoFan2017Dtrf,
  author   = {Gao, Fan},
  title    = {Distinguished theta representations for certain covering groups},
  journal  = {Pacific Journal of Mathematics},
  volume   = {290},
  number   = {2},
  pages    = {333--379},
  year     = {2017},
  doi      = {10.2140/pjm.2017.290.333}
}

@article{Mc11,
  author  = {McNamara, Peter J.},
  title   = {The metaplectic Casselman--Shalika formula},
  journal = {Transactions of the American Mathematical Society},
  volume  = {368},
  number  = {4},
  pages   = {2913--2937},
  year    = {2016},
  doi     = {10.1090/tran/6597}
}

@article{Kub67,
  author  = {Kubota, Tomio},
  title   = {Topological covering of {$\mathrm{SL}(2)$} over a local field},
  journal = {Journal of the Mathematical Society of Japan},
  volume  = {19},
  number  = {1},
  pages   = {114--121},
  year    = {1967},
  doi     = {10.2969/jmsj/01910114}
}

@article{Ginzburg2001NonvanishingOT,
  author  = {Ginzburg, David and Jiang, Dihua and Rallis, Stephen},
  title   = {Nonvanishing of the central critical value of the third symmetric power {L}-functions},
  journal = {Forum Mathematicum},
  volume  = {13},
  number  = {1},
  pages   = {109--132},
  year    = {2001},
  doi     = {10.1515/form.2001.001}
}

@book{SY,
  author    = {Sakellaridis, Yiannis and Venkatesh, Akshay},
  title     = {Periods and Harmonic Analysis on Spherical Varieties},
  series    = {Ast\'erisque},
  number    = {396},
  publisher = {Soci\'et\'e Math\'ematique de France},
  location  = {Paris},
  year      = {2017},
  pagetotal = {viii+360},
  isbn      = {978-2-85629-871-8},
  doi       = {10.24033/ast.1040}
}

@article{BBL03,
  author  = {Banks, William and Bump, Daniel and Lieman, Daniel},
  title   = {{Whittaker}--{Fourier} coefficients of metaplectic {Eisenstein} series},
  journal = {Compositio Mathematica},
  volume  = {135},
  number  = {2},
  pages   = {153--178},
  year    = {2003},
  doi     = {10.1023/A:1021763918640},
  url     = {https://doi.org/10.1023/A:1021763918640}
}

\vspace{2em}
\noindent \textsc{School of Mathematical Sciences, Capital Normal University, Beijing 100048, People's Republic of China} \\
\textit{Email address:} \texttt{sihanli314@gmail.com}

\end{document}